\documentclass[12pt, oneside]{amsart}
\usepackage{amsthm}
\usepackage{fullpage}
\usepackage[margin=1in]{geometry}
\usepackage{amsmath}
\usepackage{mathrsfs}
\usepackage{amssymb}
\usepackage{subfigure}
\usepackage{verbatim}
\usepackage{caption}
\usepackage{color}
\usepackage[dvipsnames]{xcolor}
\usepackage{enumerate}

\usepackage{array}
\usepackage{graphicx}
\usepackage{subfigure}
\newcommand{\Cov}[1]{\mathscr{#1}}
\def\HH{\Cov{H}}

\newtheorem{theorem}{Theorem}[section]
\newtheorem{proposition}[theorem]{Proposition}

\newtheorem{corollary}[theorem]{Corollary}
\newtheorem{lemma}[theorem]{Lemma}

\theoremstyle{definition}
\newtheorem{definition}{Definition}
\newtheorem{claimx}{Claim}
\newenvironment{claim}
  {\pushQED{\qed}\claimx}
  {\popQED\endclaimx}
\begin{document}

\author{Yifan Jing}
\address{%
Department of Mathematics, University of Illinois at Urbana-Champaign, Urbana, IL 61801, USA.}
\email{yifanjing17@gmail.com.}
\author{Alexandr Kostochka}
\address{%
Department of Mathematics, University of Illinois at Urbana-Champaign, Urbana, IL 61801, USA, and Sobolev Institute of Mathematics, Novosibirsk 630090, Russia.}
\email{kostochk@math.uiuc.edu.}
\author{Fuhong Ma}
\address{%
School of Mathematics, Shandong University, Jinan 250100, China.}
\email{mafuhongsdnu@163.com.}
\author{Jingwei Xu}
\address{%
Department of Mathematics, University of Illinois at Urbana-Champaign, Urbana, IL 61801, USA.}
\email{jx6@illinois.edu.}

\thanks{A.K.~was partially supported by NSF grant DMS1600592, by grants 18-01-00353A and 19-01-00682 of the Russian Foundation for Basic Research and
 by Arnold O. Beckman Campus Research Board Award RB20003 of the University of Illinois at Urbana-Champaign.}
\thanks{F.M. is corresponding author.}
\thanks{J.X.~was partially supported   by Arnold O. Beckman Campus Research Board Award RB20003 of the University of Illinois at Urbana-Champaign.}

\title{Defective DP-colorings of sparse simple graphs}

%\date{}
\maketitle

\begin{abstract}
DP-coloring (also known as correspondence coloring) is a generalization of list coloring developed recently by Dvo\v r\' ak and Postle. We introduce and study $(i,j)$-defective DP-colorings of simple graphs. Let $g_{DP}(i,j,n)$ be the minimum number of edges in an $n$-vertex DP-$(i,j)$-critical graph. In this paper we determine  sharp bound on $g_{DP}(i,j,n)$ for each $i\geq3$ and $j\geq 2i+1$  for infinitely many $n$.
\end{abstract}

\section{Introduction}
\subsection{Defective Coloring}

A \emph{proper $k$-coloring} of a graph $G$ is a partition of $V(G)$ into $k$ independent sets $V_1,\dots,V_k$.
A \emph{$(d_1, \dots, d_k)$-defective coloring} (or simply  {\em $(d_1, \dots, d_k)$-coloring}) of a graph $G$ is a partition of $V(G)$ into sets $V_1,V_2,\dots,V_k$
such that for every $i\in[k]$, every vertex in $V_i$ has at most $d_i$ neighbors in $V_i$. The ordinary proper $k$-coloring is a partial case of such
coloring, namely it
is a $(0,0,\ldots,0)$-defective coloring. A significant amount of interesting papers were devoted to
  defective colorings of graphs, see e.g.~\cite{Ar1,CCW1,DKMR,EKKOS,EMOP,HW18,KYu,LL66,OOW,VW}.

 For every $(i,j)\neq (0,0)$, it is an NP-complete problem
 to decide whether
 a graph $G$ has an $(i,j)$-coloring.  Even the problem of checking whether a given planar graph of girth $9$ has a $(0,1)$-coloring is NP-complete;
 this was showed by
 Esperet,
Montassier, Ochem, and Pinlou~\cite{EMOP}.
Since the parameter is NP-complete, a number of papers considered how sparse can be graphs with no $(i,j)$-coloring for given $i$ and $j$; the reader may look at~\cite{BIMOR10,BIMR11,BIMR12,BK11,BK14,BKY13,KKZ14,KKZ15}. Among the measures of how ``sparse" is a graph, one of the most used
 is the {\em maximum average degree},
$mad(G)=\max_{G'\subseteq G}\frac{2|E(G')|}{|V(G')|}$. In this paper we restrict ourselves to coloring with 2 colors.
A very useful notion in the studies of defective colorings with two colors is that of
 $(i,j)$-{\em critical graphs} which are
the graphs that do not have $(i,j)$-coloring but every proper subgraph of which has such a coloring. Let
$f(i,j,n)$  denote the minimum number of edges in an $(i,j)$-critical $n$-vertex graph. One simple example is that $f(0,0,n)=n$ for odd $n\geq 3$:
the $n$-cycle is not bipartite, but every graph with fewer than $n$ edges has a vertex of degree at most $1$ and hence cannot be $(0,0)$-critical.
 The  papers cited above showed several interesting bounds on $f(i,j,n)$. For example, they contain lower bounds  that are exact for infinitely many $n$
 in the cases when   $j\geq 2i+2$ and when $(i,j)\in \{(0,1),(1,1)\}$.

\subsection{Defective List Coloring}

Recall that a \emph{list} for a graph $G$ is a function $L: V(G)\rightarrow \mathcal{P}(\mathbb{N})$ that assigns to each $v\in V(G)$ a set $L(v)$.
A list $L$ is an \emph{$\ell$-list } if $|L(v)|=\ell$ for every $v\in V(G)$.
 An \emph{$L$-coloring} of $G$ is a function $\phi : V(G) \to \bigcup_{v\in V(G)} L(v)$ such that $\phi(v)\in L(v)$ for every $v\in V(G)$ and $\phi(u)\neq \phi(v)$ whenever $uv\in E(G)$. A graph $G$ is \emph{$k$-choosable} if $G$ has an $L$-coloring for every $k$-list assignment $L$. The following  notion
 was introduced in~\cite{EH1,S1999} and studied in~\cite{S2000,Woodall1,HS06,HW18}: A {\em $d$-defective list $L$-coloring}
of $G$ is a function $\phi : V(G) \to \bigcup_{v\in V(G)} L(v)$ such that $\phi(v)\in L(v)$ for every $v\in V(G)$ and every vertex has at most
$d$ neighbors of the same color. If $G$ has a $d$-defective list $L$-coloring from every $k$-list assignment $L$, then it is called
{\em $d$-defective $k$-choosable}.
As in the case of ordinary coloring, a direction of study is showing that ``sparse" graphs are $d$-defective $k$-choosable.
As mentioned before, in this paper we consider only $k=2$. The best known bounds on maximum average degree that guarantee that a
graph is $d$-defective $2$-choosable are due to Havet and Sereni~\cite{HS06} (a new proof of the lower bound is due to Hendrey and Wood~\cite{HW18}):

\bigskip
\noindent{\bf Theorem A} (\cite{HS06}){\bf.} {\em For every $d\geq 0$, if $mad(G)<\frac{4d+4}{d+2}$, then $G$ is $d$-defective $2$-choosable.
On the other hand, for every $\epsilon>0$, there is a graph $G_\epsilon$ with $mad(G_\epsilon)<4+\epsilon-\frac{ 2d+4}{d^2+2d+2}$ that is
not $(d,d)$-colorable.}

\bigskip

\subsection{Defective DP-Coloring}

 Dvo\v r\'ak and Postle \cite{DP18} introduced and studied  DP-coloring which generalizes list coloring. This notion was extended to multigraphs by Bernshteyn, Kostochka and Pron~\cite{BKP}.

%In the rest of the paper all graphs considered are simple graphs.

%we use the same definition as defined in \cite{sunflower}.
\begin{definition}\label{defn:cover}
		Let $G$ be a multigraph. A \emph{cover} of $G$ is a pair $\HH=(L, H)$, consisting of a graph $H$ (called the \emph{cover graph} of $G$) and a function $L \colon V(G) \to 2^{V(H)}$, satisfying the following requirements:
		\begin{enumerate}
			\item the family of sets $\{L(u) \,:\,u \in V(G)\}$ forms a partition of $V(H)$;
			\item for every $u \in V(G)$, the graph $H[L(u)]$ is complete;
			\item if $E(H[L(u), L(v)]) \neq \varnothing$, then either $u = v$ or $uv \in E(G)$;
				\item \label{item:matching} if the multiplicity of an edge $uv \in E(G)$ is $k$, then $H[L(u), L(v)]$ is the
union  of at most $k$  matchings connecting $L(u)$ with $L(v)$.
		\end{enumerate}
		A cover $(L, H)$ of $G$ is \emph{$k$-fold} if $|L(u)| = k$ for every $u \in V(G)$.
	\end{definition}
		 %of simple graphs. If not stated otherwise, ``graph'' means ``simple graph''.

Throughout this paper, we consider only $2$-fold covers.
				
	 \begin{definition}
		Let $G$ be a multigraph and $\HH=(L, H)$ be a cover of $G$. An \emph{$\HH$-map} is
	an injection $\phi: V(G)\rightarrow V(H)$, such that $\phi(v)\in L(v)$ for every $v\in V(G)$. The subgraph of $H$ induced by $\phi(V(G))$ is called the \emph{$\phi$-induced cover graph}, denoted by $H_{\phi}$.
	 \end{definition}
	
   \begin{definition}\label{ijcolor} Let $\HH=(L, H)$ be a cover of $G$. For $u \in V(G)$, let $L(u)=\{p(u), r(u)\}$, where $p(u)$ and $r(u)$ are called the \emph{poor} and the \emph{rich} vertices, respectively.	Given $i, j\geq 0$ and $i\leq j$.  An $\HH$-map $\phi$ is an \emph{$(i, j)$-defective-$\HH$-coloring of $G$} if the degree of every poor vertex in $H_\phi$ is at most $i$, and the degree of every rich vertex in $H_\phi$ is at most $j$.
   %We say that $G$ is \emph{$(i,j)$-defective-DP-colorable} if for every $2$-fold cover $(L,H)$ of $G$, there exists an $(i, j)$-defective-$H$-coloring.
   \end{definition}

	\begin{definition}\label{D-S-1} A multigraph $G$ is \emph{$(i,j)$-defective-DP-colorable} if for every $2$-fold cover $\HH=(L,H)$ of $G$, there exists an $(i, j)$-defective-$\HH$-coloring. We say $G$ is \emph{$(i,j)$-defective-DP-critical}, if $G$ is not $(i,j)$-defective-DP-colorable, but every proper subgraph of $G$ is.
\end{definition}

  If  $uv\in E(G)$ and  in a $2$-fold cover $\HH=(L,H)$ of  $G$  some vertex
 $\alpha\in L(u)$ has no neighbors in $L(v)$, then also some $\beta\in L(v)$ has no neighbors in $L(u)$. In this case,  adding $\alpha\beta$
to $H$ makes it only harder to find an $(i, j)$-defective-$\HH$-coloring of $G$. Thus, below we consider only {\em full} $2$-fold covers, i.e. the covers
$\HH=(L,H)$ of $G$ such that for every edge $e$ connecting $u$ with $v$ in $G$, the matching in $\HH=(L,H)$ corresponding to $e$ consists of two edges.

   For brevity, in the rest of the paper, we call an $(i, j)$-defective-$\HH$-coloring simply by an \emph{$(i,j,\HH$)-coloring} (or `$\HH$-coloring', if
   $i$ and $j$ are clear from the context). Similarly, instead of ``$(i, j)$-defective-DP-colorable" and ``$(i,j)$-defective-DP-critical'' we will say  ``$(i,j)$-colorable" and ``$(i,j)$-critical".

  Denote the minimum number of edges in an $n$-vertex \emph{$(i,j)$-critical multigraph} by $f_{DP}(i,j,n)$, and
 the minimum number of edges in an $n$-vertex \emph{$(i,j)$-critical simple graph} by $g_{DP}(i,j,n)$. By definition,
 $  f_{DP}(i,j,n)\leq g_{DP}(i,j,n)$.
  Recently~\cite{JKMSX}, linear lower bounds on $f_{DP}(i,j,n)$ were proved that are exact for  infinitely many $n$
  for every choice of $i\leq j$.
  %:\textcolor{blue}{in abstract we use $f_{DP}(i,j,n)$, but here we use $g_{DP}(i,j,n)$ for our result in this paper, might be a little confusing}
  %The goal of this paper is to find exact linear lower bounds of $f_{DP}(i,j,n)$ for simple graphs. 		
\bigskip

\noindent{\bf Theorem B} (\cite{JKMSX}){\bf.} {\em
%\begin{theorem}\label{multi}
\begin{enumerate}
        \item If $i = 0$ and $j\geq 1$, then $f_{DP}(0,j,n) \geq n+j$. This is sharp for every $j\geq 1$       and every $n\geq 2j+2$.
        \item If $i\geq1$ and $j\geq 2i+1$, then $f_{DP}(i,j,n)\geq \frac{(2i+1)n-(2i-j)}{i+1}$. This is sharp for each such pair $(i,j)$ for infinitely many $n$.
        \item If $i\geq1$ and $i+2\leq j\leq 2i$, then $f_{DP}(i,j,n)\geq \frac{2jn+2}{j+1}$. This is sharp for each such pair $(i,j)$ for infinitely many $n$.
     %   Moreover, there exist infinitely many $(i,j)$-critical graphs $H$ with $|E(H)|=\frac{2j|V(H)|+2}{j+1}$.
        \item If $i\geq1$, then $f_{DP}(i,i+1,n)\geq \frac{(2i^2+4i+1)n+1}{i^2+3i+1}$. This is sharp for each  $i\geq 1$ for infinitely many $n$.
   %     Moreover, there exist infinitely many $(i,i+1)$-critical graphs $H$ with $|E(H)|=\frac{(2i^2+4i+1)|V(H)|+1}{i^2+3i+1}$.
        \item If $i\geq1$,  then $f_{DP}(i,i,n)\geq \frac{(2i+2)n}{i+2}$. This is sharp for each $i\geq 1$ for infinitely many $n$.
     %   Moreover, there exist infinitely many $(i,i)$-critical graphs $H$ with $|E(H)|=\frac{(2i+2)|V(H)|}{i+2}$.
    \end{enumerate}
%\end{theorem}
The bound in Part (1) is also sharp for simple graphs. }

\bigskip
For $i>0$ we do not know simple graphs for which the bounds of Theorem~B are sharp. In fact, we think that $g_{DP}(i,j,n)>f_{DP}(i,j,n)$ for $i>0$.
It follows from~\cite{KX} that $g_{DP}(1,1,n)>f_{DP}(1,1,n)$ and $g_{DP}(2,2,n)>f_{DP}(2,2,n)$.
The goal of this paper is to find a lower bound on $g_{DP}(i,j,n)$ for $i\geq 3$ and $j\geq 2j+1$ that is exact for infinitely many $n$ for each such pair $(i,j)$.
It differs from the bound of Theorem B(2) but only by $1$.

\section{Results}

The goal of this paper is to prove the following extremal result.

\begin{theorem}\label{MT-f}
Let $i \geq 3$, $j \geq 2i+1$ be positive integers, and let $G$ be an $(i,j)$-critical simple graph.
Then $$g_{DP}(i,j,n) \geq \frac{(2i+1)n+j-i+1}{i+1}.$$
This is sharp for each such pair $(i,j)$ for infinitely many $n$.
\end{theorem}

Since every non-$(i,j)$-colorable graph contains an $(i, j)$-critical subgraph, Theorem~\ref{MT-f} yields the following.

\begin{corollary}
Let $G$ be a simple graph. If $i\geq3$ and $j\geq 2i+1$ and for every subgraph $H$ of $G$, $|E(H)| \leq \frac{(2i+1)|V(H)|+j-i+1}{i+1},$ then $G$ is $(i,j)$-colorable. This is sharp.
\end{corollary}

In the next section we introduce a more general framework to prove the lower bound of Theorem~\ref{MT-f}.
The lower bound of Theorem~\ref{MT-f} will be proved in Section~4. In the last section, we present constructions showing that our bounds are sharp for each $i\leq j$ for infinitely many $n$.

\section{A More General Setting}

We need the following more general
 framework. Instead of $(i,j)$-colorings of a cover $(L,H)$ of a graph $G$, we will consider $(L,H)$-maps $\phi$ with variable restrictions
 on the `allowed' degrees of the vertices in $H_\phi$.

\begin{definition}[Capacity]\label{M-D-1}
A \emph{capacity function} on $G$ is a map ${\bf c}: V(G)\rightarrow \{-1,0,\dots,i \}\times \{-1,0,\dots,j\}$. For $u\in V(G)$, denote ${\bf c}(u)$ by $({\bf c}_1(u), {\bf c}_2(u))$.
%denote ${\bf c}(u)$ by $({\bf c}_1(u), {\bf c}_2(u))$, where ${\bf c}_1(u)$ and ${\bf c}_2(u)$ are the capacity of $p(u)$ and $r(u)$, respectively.
We call such pair $(G,{\bf c})$ a \emph{weighted pair}.
\end{definition}

 %For $S\subseteq V(G)$, we denote ${\bf c}|_S$ by ${\bf c}$ for simplicity.
 Below, let $(G,{\bf c})$ be a weighted pair, and $\HH=(L,H)$ be a cover of $G$.

\begin{definition}[A ${\bf c}$-coloring]\label{ijc}
A \emph{$({\bf c},\HH$)-coloring of $G$} is an $\HH$-map $\phi$ such that for each $u\in V(G)$, the degree of $p(u)$ in $H_{\phi}$ is at most ${\bf c}_1(u)$, and that of $r(u)$ is at most ${\bf c}_2(u)$.
We call ${\bf c}_1(u)$ the \emph{capacity} of $p(u)$  and ${\bf c}_2(u)$ the \emph{capacity} of $r(u)$.
If the capacity of some $v$ in $V(H)$ is $-1$, then $v$ is not allowed in the image of any $({\bf c},\HH$)-coloring of $G$. If for every cover $\HH$ of $G$, there is a $({\bf c},\HH$)-coloring, we say that $G$ is {\em ${\bf c}$-colorable}.
\end{definition}

If ${\bf c}(v) = (i,j)$ for all $v\in V(G)$, then any $( {\bf c},\HH$)-coloring of $G$ is an  $(i,j,\HH$)-coloring in the  sense of Definition~\ref{ijcolor}. So, Definition~\ref{ijc} is
a refinement of Definition~\ref{ijcolor}. Similarly, we say that $G$ is ${\bf c}$-{\em critical} if $G$ is not ${\bf c}$-colorable, but every proper subgraph of $G$ is.
For every vertex $x$ in the cover graph, we slightly abuse the notation of ${\bf c}$ and denote the capacity of $x$ by ${\bf c}(x)$.

%From now on, we consider a more general setting. That is, we associate a pair $({\bf c}_1(v),{\bf c}_2(v))$ to each vertex $v$. Let $\phi$ be an representative map. We say \emph{$G$ is colorable} if for every $v$, in the $\phi$-induced cover graph $H_\phi$, the degree of $p(v)$ is at most ${\bf c}_1(v)$, and the degree of $r(v)$ is at most ${\bf c}_2(v)$. We will prove a stronger version of our main results in this general setting. To recover the $(i,j)$-defective-DP-coloring, we just let ${\bf c}_1(v)=i$ and ${\bf c}_2(v)=j$ for every $v\in V(G)$.

%In the rest of the paper we use $({\bf c}_1(u), {\bf c}_2(u))$ to denote the capacity of a vertex $u \in V(G)$.

\begin{definition}\label{M-D-2} For a vertex $u \in V(G)$, the \emph{$(i,j, {\bf c})$-potential} of $u$ is $$\rho_{{\bf c}}(u):= i-j+1+{\bf c}_1(u)+{\bf c}_2(u).$$
The $(i,j, {\bf c})$-potential of a subgraph $G'$ of $G$ is
\begin{equation}\label{eq-1}
\rho_{G,{\bf c}}(G') := \sum_{u \in V(G')}\rho_{{\bf c}}(u)-(i+1)|E(G')|.
\end{equation}
For a subset $S\subseteq V(G)$, the $(i,j, {\bf c})$-potential of $S$, $\rho_{G,{\bf c}}(S)$, is the $(i,j, {\bf c})$-potential of $G[S]$.
The $(i,j, {\bf c})$-potential of $(G,{\bf c})$ is defined by $\rho(G,{\bf c}) := \min_{S\subseteq V(G)}\rho_{G,{\bf c}}(S)$.
\end{definition}
When clear from the text, we call the $(i,j, {\bf c})$-potential simply by potential.

Observe that the potential function is submodular:
\begin{lemma}\label{lem:submod}
For all $A,B\subseteq V(G)$,
\begin{equation}\label{eq:submod}
    \rho_{G,{\bf c}}(A)+\rho_{G,{\bf c}}(B)=\rho_{G,{\bf c}}(A\cup B)+\rho_{G,{\bf c}}(A\cap B)+(i+1)|E(A\setminus B,B\setminus A)|.
\end{equation}
\end{lemma}
\begin{proof}
Since $G$ and ${\bf c}$ are fixed, we omit the subscripts in the proof. By definition,
\begin{align*}
    &\rho(A)=\rho(A\setminus B)+\rho(A\cap B)-(i+1)|E(A\setminus B,A\cap B)|,\\
    &\rho(B)=\rho(B\setminus A)+\rho(A\cap B)-(i+1)|E(B\setminus A,A\cap B)|.
\end{align*}
Hence
\begin{align*}
    \rho(A)+\rho(B)=\rho(A\cap B)+\rho(A\cup B)+(i+1)|E(A\setminus B,B\setminus A)|\geq \rho(A\cap B)+\rho(A\cup B).
\end{align*}
%finishes the proof.
\end{proof}

The following theorem implies the lower bound in Theorem \ref{MT-f}.

%we restate Theorem \ref{MT-f} using potential language as follows.

\begin{theorem}\label{MT-F}
Let $i \geq 3$, $j \geq 2i+1$ be positive integers, and $(G,{\bf c})$ be a weighted pair such that $G$ is  ${\bf c}$-critical, then $\rho(G,{\bf c}) \leq i-j-1$.
\end{theorem}
To deduce the lower bound in Theorem \ref{MT-f}, simply set ${\bf c}(v)=  (i,j)$ for every $v\in V(G)$. We will prove Theorem~\ref{MT-F} in the next section.

%For a vertex $u$ with $d(u)=2$ with either ${\bf c}_1(u) \leq 1$ or ${\bf c}_2(u) \leq 1$, we have
%\begin{equation}\label{eq-b'}
%\rho_{{\bf c}}(u)=i-j+1+{\bf c}_1(u)+{\bf c}_2(u) \leq i-j+1+1+j \leq i+2.\end{equation}

%Similarly, if $d(u) = 3$ and with either ${\bf c}_1(u) \leq 2$ or ${\bf c}_2(u) \leq 2$, then
%\begin{equation}\label{eq-c}\rho_{{\bf c}}(u)\leq i+3.    \end{equation}

%Suppose $x$ is a vertex in $G$ with $N(x)=\{u,v\}$.
%We say $x$ is \emph{even} (with respect to $\HH$), if in $H$, each vertex in $L(x)$ is adjacent either to both rich vertices in $L(u)$ and $L(v)$, or to both poor vertices in $L(u)$ and $L(v)$. Otherwise, we say $x$ is \emph{odd}.

%\section{Proof of Main Theorem}
\section{Proof of Theorem~\ref{MT-F}}
 Suppose there exists a ${\bf c}$-critical graph $G$ with
$\rho(G, {\bf c}) \geq i-j$.  Choose such $(G,{\bf c})$ with $|V(G)|+|E(G)|$ minimum. We say that $G'$ is \emph{smaller} than $G$ if $|V(G')|+|E(G')|< |V(G)|+|E(G)|$. Let $\HH=(L,H)$ be an arbitrary cover of $G$.

For a subgraph $G'$ of $G$, let $\HH_{G'} = (L_{G'}, H_{G'})$ denote the subcover \emph{induced} by $G'$, i.e.,
\\ (1) $L_{G'} = L|_{V(G')}$, where `$f|_S$' is the restriction of function $f$ to subdomain $S$; \\
(2) $V(H_{G'}) = L(V(G'))$ and $L_{G'}(v) = L(v)$ for every $v\in V(G')$;\\
 (3) $H_{G'}[L(u)\cup L(v)] = H[L(u)\cup L(v)]$ for every $uv\in E(G')$, and for $x,y$ such that $xy\notin E(G')$, there is no edge between $L_{G'}(x)$ and $L_{G'}(y)$.

For a subset $S$ of $V(G)$, let $\HH_S = (L_S, H_S)$ denote the subcover induced by $G[S]$. If a capacity function is the restriction of ${\bf c}$ to some $S\subseteq V(G)$, we denote this capacity function by ${\bf c}$ instead of ${\bf c}|_S$, for simplicity.

For two vertices $x,y$, we use \emph{$x\sim y$} to indicate  that $x$ is adjacent to $y$, and $x\nsim y$ to indicate that $x$ is not adjacent to $y$.

\begin{lemma}\label{LM-M-1} Let $S$ be a proper subset of $V(G)$. If $\rho_{G,{\bf c}}(S) \leq i-j$, then $S=\{x\}$ for some $x\in V(G)$ with $\rho_{{\bf c}}(x)=i-j$.
\end{lemma}
\begin{proof}
Suppose the lemma fails. Let $S$ be a maximal proper subset of $V(G)$ such that $\rho_{G,{\bf c}}(S)\leq i-j$ and $|S|\geq 2$.
%Let $H$ be an arbitrary cover graph of $G$.
If $|N(v)\cap S|\geq 2$ for some $v\in V(G)\setminus S$, then
\[
\rho_{G,{\bf c}}(S\cup\{v\})\leq  i-j+2i+1-2(i+1) = i-j-1.
\]
If $S\cup \{v\}\neq V(G)$, this contradicts the maximality of S, otherwise this contradicts the choice of $G$.  Thus
\begin{equation}\label{new1}
\mbox{\em $|N(v)\cap S|\leq 1$ for every $v\in V(G)\setminus S$.}
\end{equation}
Since $G$ is ${\bf c}$-critical,  $G[S]$ admits an $({\bf c},\HH_S)$-coloring $\phi$.

Construct $G'$ from $G-S$ by adding a new vertex $v^*$ adjacent to every $u\in V(G)-S$ that was adjacent to a vertex in $S$.
Define a capacity function ${\bf c'}$ by letting
 ${\bf c'}(v^*) = (-1,0)$ and ${\bf c'}(u) = {\bf c}(u)$ for $u\in V(G'-v^*)$.

 By~\eqref{new1}, $G'$ is simple. Suppose $\rho_{G',{\bf c'}}(A)\leq i-j-1$ for some $A\subseteq V(G')$. Since $G'-v^*\subseteq G$ and
 ${\bf c'}(u) = {\bf c}(u)$ for $u\in V(G'-v^*)$, $v^*\in A$. Then using~\eqref{eq:submod} and $\rho_{G,{\bf c}}(S)\leq i-j=\rho_{G',{\bf c'}}(v^*)$,
 $$\rho_{G,{\bf c}}(S\cup (A-v^*))=\rho_{G,{\bf c}}(S)+\rho_{G,{\bf c}}(A-v^*)-(i+1)|E_G(S,A-v^*)|
 $$
 $$\leq \rho_{G',{\bf c'}}(v^*)+\rho_{G',{\bf c'}}(A-v^*)-(i+1)|E_{G'}(v^*,A-v^*)|=\rho_{G',{\bf c'}}(A)\leq i-j-1.
 $$
 Again, this contradicts either the maximality of $S$ or the choice of $G$. This yields
 \begin{equation}\label{new2}
\rho(G',{\bf c'})\geq i-j.
\end{equation}

For every $x\in S$ and $y\in N(x)\setminus S$, denote the neighbor of $\phi(x)$ in $L(y)$ by $y_{\phi}$. Let $\HH' = (L',H')$ be the cover of $G'$ defined as follows:

1) $L'(v^*) = \{p(v^*), r(v^*) \}$, and $L'(u) = L(u)$ for every $u\in V(G)\setminus S$;

2) $y_{\phi}\sim r(v^*)$ for every $y\in N(S)$,  and    $H'[\{u,w\}] = H[\{u,w\}]$ for every $u,w \in V(G'-v^*)$.

By~\eqref{new2} and the minimality of $G$, $G'$ has a $({\bf c'}, \HH')$-coloring $\psi$. Since ${\bf c'}(v^*) = (-1,0)$,
\begin{equation}\label{new3}
\mbox{\em $\psi(v^*)=r(v^*)$ and
$\psi(y)\neq y_{\phi}$ for every $y\in N(S)$.
}
\end{equation}

Let $\theta$ be an $\HH$-map such that $\theta\mid _S = \phi$ and $\theta\mid _{V(G)\setminus S} = \psi\mid_{V(G'-v^*)}$.
By~\eqref{new3},  $\theta$ is a $({\bf c}, \HH)$-coloring of $G$, a contradiction.
\end{proof}

Lemma~\ref{LM-M-1} implies that \begin{equation}\label{eq-a}
\text{for every  $F\subseteq V(G)$, }  \rho_{G,{\bf c}}(F) \geq i-j.
\end{equation}

\begin{lemma}\label{LM-M-2}For every $u \in V(G)$, the following statements hold:

\smallskip
\emph{(i)} ${\bf c}_1(u), {\bf c}_2(u) \geq 0$;\hspace{9mm}
%\smallskip
\emph{(ii)} $d_G(u) \geq 2$;\hspace{9mm}
%\smallskip
\emph{(iii)} $\rho_{{\bf c}}(u) \geq i-j+1$.
\end{lemma}
\begin{proof} We prove (i) by contradiction. Suppose there is a vertex $u \in V(G)$ with $L(u) = \{\alpha(u), \beta(u)\}$, where ${\bf c}(\alpha(u))=-1$.
Let $v \in N_G(u)$ and  $L(v)=\{\alpha(v), \beta(v)\}$, where $\alpha(v)\alpha(u), \beta(v)\beta(u)\in E(H)$.
%Note that if $c(\beta(v)) = -1$, then $G-u$ has a coloring $\phi$ with $\phi(v) = \alpha(v)$. Extend $\phi$ to $u$ by $\phi(u) = \beta(u)$, then $\phi$ is an $(i,j)$-coloring on $G$, a contradiction. Hence we assume in the following that $\beta(v)\geq 0$.

\medskip

\noindent{\bf Case 1.} ${\bf c}(\beta(u)) = 0$. If $\min\{{\bf c}_1(v), {\bf c}_2(v)\}=-1$, then
 $$\rho_{G,{\bf c}}(\{u,v\})=\rho_{{\bf c}}(u)+\rho_{{\bf c}}(v)-(i+1)\leq (i-j)+(i-j+1-1+j)-(i+1)= i-j-1,$$
  a contradiction to (\ref{eq-a}).
Thus  ${\bf c}_1(v), {\bf c}_2(v) \geq 0$.

For every $w\in N_G(v)$, let $L(w) = \{\alpha(w), \beta(w)\}$, so that $\alpha(w)\sim \alpha(v), \beta(w)\sim \beta(v)$.
Since $G$ is ${\bf c}$-critical, graph $G-v$ has a $({\bf c}, \HH_{G-v})$-coloring $\phi$.

\noindent{\bf Case 1.1:}
 ${\bf c}(\alpha(w)) = -1$ for all $w\in N(v)$ (in particular, this happens if $d(v)=1$).
 %Let $\phi$ be a $({\bf c},\HH_{G-v})$-coloring of $G-v$, t
 Then $\phi(w) = \beta(w)$ for every $w\in N(v)$. Extend $\phi$ to $v$ by  $\phi(v) = \alpha(v)$. This  $\phi$ is a $({\bf c},\HH)$-coloring of $G$, a contradiction.

\noindent{\bf Case 1.2:} There is $w\in N(v)$ such that ${\bf c}(\alpha(w))\geq 0$. Then $\rho_{{\bf c}}(w) \geq i-j+1$ since otherwise $\rho_{G,{\bf c}}(\{u,v,w\})\leq 2(i-j)-1<i-j-1$.
Define $(G', {\bf c'})$  as follows:
\\
1) $G' = G-vw$ and $\HH' = (L,H')$ is the sub-cover of $\HH$ induced by $G'$;\\
2) ${\bf c'}$ differs from ${\bf c}$ only for $\alpha(v)$ and $\alpha(w)$:
${\bf c'}(\alpha(v)) = {\bf c}(\alpha(v))-1$  and ${\bf c'}(\alpha(w))={\bf c}(\alpha(w))-1$.
%Let $\HH' = (L,H')$ be the sub-cover induced by $G'$.

By the minimality of $G$, if $G'$ is not colorable, then there is $F\subseteq V(G')$  with $\rho_{G,{\bf c'}}(F) \leq i-j-1$. By~\eqref{eq-a},
$F\cap \{v,w\}\neq \emptyset$.
If $v, w \in F$, then
$$\rho_{G,{\bf c}}(F)=\rho_{G,{\bf c'}}(F)+2-(i+1) < \rho_{G,{\bf c'}}(F) < i-j.$$ If $v \in F$ and $w \notin F$, then $\rho_{G,{\bf c}}(V(F)) \leq \rho_{G,{\bf c'}}(F)+1 \leq i-j$. Since $\rho_{{\bf c}}(v)\geq i-j+1$, this  contradicts Lemma \ref{LM-M-1}. Since $\rho_{{\bf c}}(w)\geq i-j+1$, the case when $w\in F, v\notin F$ is impossible for the same reason.  Thus, $G'$ admits a $({\bf c'},\HH')$-coloring $\phi$.
Since ${\bf c}(\alpha(u))=-1$ and ${\bf c}(\beta(u)) = 0$,  we have $\phi(u) = \beta(u)$ and $\beta(u)$ has no neighbors
%\textcolor{orange}{$d(\phi(u)) = 0$
 in $H'_{\phi}$.
 Then $\phi(v) = \alpha(v)$ and by the construction of $G'$, independently of the color of $w$, $\phi$ is a $({\bf c}, \HH)$-coloring of $G$.

\medskip

\noindent{\bf Case 2.} ${\bf c}(\beta(u))\geq 1$.
Then $\rho_{{\bf c}}(u)\geq i-j+1$. Since $G$ is ${\bf c}$-critical,  $G-uv$ has a $({\bf c}, \HH_{G-uv})$-coloring $\phi$.
If ${\bf c}(\beta(v)) = -1$, then $\phi(u) = \beta(u)$ and $\phi(v) = \alpha(v)$. So, $\phi$ is also a $({\bf c}, \HH)$-coloring of $G$, a contradiction. Hence ${\bf c}(\beta(v))\geq 0$. Also by Lemma~\ref{LM-M-1},
$$i-j+1\leq\rho_{G,{\bf c}}(\{u,v\})=\rho_{{\bf c}}(v)+\rho_{{\bf c}}(u)-(i+1)\leq \rho_{{\bf c}}(v)+i-j+1-1+j-i-1,$$ thus $\rho_{{\bf c}}(v)\geq i-j+2.$

Define $(G', {\bf c'})$  as follows:
\\
1) $G' = G-vu$ and $\HH' = (L,H')$ is the sub-cover of $\HH$ induced by $G'$;\\
2) ${\bf c'}$ differs from ${\bf c}$ only for $\beta(v)$ and $\beta(u)$:
${\bf c'}(\beta(v)) = {\bf c}(\beta(v))-1$  and ${\bf c'}(\beta(u))={\bf c}(\beta(u))-1$.

Repeating the argument of Case 1.2, we prove that $G'$ has a $({\bf c'}, \HH_{G'})$-coloring $\psi$, which is also a $({\bf c}, \HH)$-coloring of $G$, a contradiction. This proves~(i).

\medskip

For (ii), suppose there is a vertex $u$ with $N(u) = \{v\}$. By (i), ${\bf c}_1(u), {\bf c}_2(u) \geq 0$.
 Since $G$ is ${\bf c}$-critical,
  $G-u$ has a $({\bf c}, \HH_{G-u})$-coloring $\phi$. Now choosing $\phi(u)\in L(u)$ not adjacent to $ \phi(v)$ we obtain a $({\bf c}, \HH)$-coloring of $G$,
  a contradiction. This proves (ii), and
(iii) follows immediately from (i).
\end{proof}

Lemmas~\ref{LM-M-1} and \ref{LM-M-2} (iii) imply that
\begin{equation}\label{eq-b}
\text{for every  $\emptyset \neq F\subsetneq V(G)$, } \rho_{G,{\bf c}}(F) \geq i-j+1.
\end{equation}

We say a vertex $v\in V(G)$ is a \emph{$(d;c_1,c_2)$-vertex} if $d(v)=d$, ${\bf c}_1(v)=c_1$, and ${\bf c}_2(v)=c_2$.
%Let $A\subseteq \{0,\dots,n\}$, $B\subseteq\{-1,\dots,i\}$, and $C\subseteq\{-1,\dots,j\}$, $v$ is a \emph{$(A;B,C)$-vertex} if $d(v)\in A$, ${\bf c}_1(v)\in B$ and ${\bf c}_2(v)\in C$.
\iffalse
\begin{definition}
A vertex $v\in V(G)$ is a \emph{$(d;c_1,c_2)$-vertex} if $d(v)=d$, ${\bf c}_1(v)=c_1$, and ${\bf c}_2(v)=c_2$. Let $A\subseteq \{0,\dots,n\}$, $B\subseteq\{-1,\dots,i\}$, and $C\subseteq\{-1,\dots,j\}$, a vertex $v$ is an \emph{$(A;B,C)$-vertex} if $d(v)\in A$, ${\bf c}_1(v)\in B$ and ${\bf c}_2(v)\in C$.
\end{definition}
\fi
The following lemma is a crucial ingredient of our argument.

\begin{lemma}\label{LM-M-3} Let $\emptyset\neq F\subsetneq V(G)$. If $\rho_{G,{\bf c}}(F) \leq i-j+1$, then $V(G)\setminus F=\{x\}$, where $x$ is a $(2;i,j)$-vertex.
\end{lemma}
\begin{proof}
Suppose the lemma fails. Then there is  a maximal  $\emptyset\neq F\subsetneq V(G)$ such that $\rho_{G,{\bf c}}(F)\leq i-j+1$ and $V(G)-F$ is not a single $(2;i,j)$-vertex.

If $V(G)-F = \{v\}$, then by the choice of $F$ either $\rho(v)<2i+1$ or by Lemma~\ref{LM-M-2} (ii), $d(v)\geq 3$. In both cases,
 by~\eqref{eq:submod},
 $$\rho(V(G))=\rho(F)+\rho(v)-(i+1)d(v)<(i-j+1)+(2i+1)-2(i+1)=i-j,$$
  a contradiction. Hence $|V(G\setminus F)|\geq 2$.
Because of~\eqref{eq-b}, we have
\begin{equation}\label{nbg}
    |N(v)\cap F|\leq 1,\,\text{ for every } v\in V(G)\setminus F.
\end{equation}
Let $Y$ be the set  of all $(2;i,j)$-vertices in $V(G)-F$, and $X=V(G)\setminus F\setminus Y$.
\begin{claim}\label{clm:1}
{\em Both $X$ and $Y$ are independent sets.}
\medskip

\noindent\emph{Proof of Claim~\ref{clm:1}.}  Suppose $u, v \in Y$ and $u\sim v$. Let $u'\in N(u)-v, v'\in N(v)-u$.
Since $G$ is ${\bf c}$-critical,
 $G-\{u,v\}$ has a $({\bf c}, \HH_{G-\{u,v\}})$-coloring $\phi$. We extend $\phi$ to $u$ and $v$ by choosing $\phi(u)\in L(u)$ with
 $\phi(u)\nsim \phi(u')$ and $\phi(v)\in L(v)$ with
  $\phi(v)\nsim \phi(v')$. Since $u$ and $v$ are $(2;i,j)$-vertices,
the new $\phi$ is a $({\bf c},\HH)$-coloring of $G$, a contradiction.

Now suppose $u, v \in X$ and $u\sim v$. Let $G' = G-uv$. Define ${\bf c'}$ by $ {\bf c'}(y) = {\bf c}(y)$ for $y\notin \{u,v\}$ and
${\bf c'_k}(x) = {\bf c_k}(x)-1$ for $x\in \{u,v\}$ and $k\in \{1,2\}$.
If $G'$ has a $({\bf c'},\HH)$-coloring $\phi$, then $\phi$ is a $({\bf c},\HH)$-coloring of $G$, a contradiction. Thus $G'$ has no such coloring.
By the minimality of $G$, this yields that
 there is  $Q\subseteq V(G')$ with $\rho_{G',{\bf c'}}(Q) < i-j$. By the construction of $G'$,  $Q\cap \{u,v\}\neq\emptyset$.
 If $\{u,v\} \subseteq Q$, then $$\rho_{G,{\bf c}}(Q) = \rho_{G',{\bf c'}}(Q)+4-(i+1) < i-j,$$ a contradiction.
 Thus by the symmetry between $u$ and $v$ we may assume  $u \in V(Q)$ and $v \notin V(Q)$.
 In this case,
$$\rho_{G,{\bf c}}(Q) \leq \rho_{G',{\bf c'}}(Q)+2 \leq i-j-1+2=i-j+1.$$
Note  that by maximality of $F$, $F\setminus Q\neq\varnothing$ and $F\cap Q\neq\varnothing$. Then by (\ref{eq:submod}) and (\ref{eq-b}),
\[
\rho_{G,{\bf c}}(F\cup Q)\leq 2(i-j+1)-\rho_{G,{\bf c}}(F\cap Q)\leq i-j+1.
\]
 Since $v\notin F\cup Q$ and $v$ is not a $(2;i,j)$-vertex, this contradicts the maximality of $F$.
% \hfill $\bowtie$
\end{claim}

\medskip
%For $S\subseteq V(G)$, let $\partial S := N_G(S)$. Similary for a subgraph $G'$ of $G$, $\partial G' := \partial V(G')$.
Let $X_1:=N_G(F)\cap X$, $Y_1:=N_G(F)\cap Y$, $X_0:=X\setminus X_1$, and $Y_0:=Y\setminus Y_1$.
%Denote $H[V(F)]$ by $H^F$, and the subcover of $\HH$ on $F$ by $\HH_F$.

\begin{claim}\label{clm:2}
{\em Let $u\in F\cap N_G(X\cup Y)$. For any $({\bf c},\HH_F)$-coloring $\phi$ of $G[F]$, the degree of $\phi(u)$ in the $\phi$-induced subgraph $H_\phi$ is equal to ${\bf c}(\phi(u))$.}
%$d(\phi(u)) = c(\phi(u))$ in $H^F_{\phi}$.
\medskip

\noindent\emph{Proof of Claim~\ref{clm:2}.}
Suppose that for some $({\bf c},\HH_F)$-coloring $\phi$ of $G[F]$, the degree of $\phi(u)$ in $H_\phi$ is at most
$ {\bf c}(\phi(u))-1$. Let $w\in N(u)\setminus F$. Denote $S:=X_1\cup Y_1\setminus \{w\}$.
%Let $H_\phi$ be the $\phi$-induced cover graph. Suppose $\phi(u)$ has capacity $c$, and $|N_{H_\phi}(\phi(u))|\leq c-1$. Let $w$ be a neighbor of $u$ in $G-H$. Consider two cases.
\medskip

\noindent{\bf Case 1:} $w\in Y_1$.
%\medskip
Let $w'$ be the other neighbor of $w$. By Claim~\ref{clm:1} and (\ref{nbg}), $w'\in X$.
Construct $G'$ from $G-F-w$ by adding a new vertex $v^*$ adjacent to each vertex in $S$.
By~\eqref{nbg},
\begin{equation}\label{new5}
\mbox{\em $|E_G(F,S')|=|E_{G'}(v^*,S')|$ for every $S'\subseteq S$. }
\end{equation}
Define ${\bf c'}$ by ${\bf c'}(x) = {\bf c}(x)$ for all $x\in V(G')\setminus \{v^*\}$ and
 ${\bf c'}(v^*) = (0,-1)$.
 Define $L'$ by $L'(v^*)= \{p(v^*),r(v^*) \}$ and  $L'(x) = L(x)$ for $x\in V(G')- v^*$.
Let $\HH'=(L',H')$ be a cover of $G'$ such that $H'[\{x,y\}]= H[\{x,y\}]$ for $x,y\in V(G')-v^*$
and the neighbors of $p(v^*)$ and $r(v^*)$ are defined as follows. For each $v\in S$, if $z\in N(v)\cap F$ and $L(v)=\{\alpha(v),\beta(v)\}$ where $\alpha(v)\sim
\phi(z)$, then $p(v^*)\sim \alpha(v)$ and $r(v^*)\sim \beta(v)$.

If there is $Q\subset V(G')$ with $\rho_{G',{\bf c'}}(Q)\leq i-j-1$, then $v^*\in Q$ since $G'-v^*\subset G$.
 In this case,  using~\eqref{eq:submod}
and~\eqref{new5} and remembering that $\rho_{G',{\bf c'}}(v^*)=i-j$,
$$
\rho_{G,{\bf c}}((Q-v^*)\cup F)= \rho_{G,{\bf c}}(F)+\rho_{G,{\bf c}}(Q-v^*)-(i+1)|E_G(F,Q-v^*)|$$
$$\leq (i-j+1)+\rho_{G',{\bf c'}}(Q-v^*)-(i+1)|E_{G'}(v^*,Q-v^*)|=1+\rho_{G',{\bf c'}}(Q)\leq i-j.
$$
Since $w\notin (Q-v^*)\cup F$, this contradicts~\eqref{eq-b}. Thus $\rho(G',{\bf c'})\geq i-j$.
Hence by the minimality of $G$, $G'$ has a $({\bf c'}, \HH')$-coloring $\psi$. Since ${\bf c'}(v^*) = (0,-1)$,
$\psi(v^*)=p(v^*)$ and the degree of $p(v^*)$ in $\HH'_\psi$ is zero. Hence
\begin{equation}\label{new6}
\mbox{\em $\psi(v)=\beta(v)$ for every $v\in S$. }
\end{equation}

Define an $\HH$-map $\theta$ of $G$ by $\theta(x) = \phi(x)$ for $x\in F$,  $\theta(v) = \psi(v)$ for $v\in V(G)-F-w$ and choosing $\theta(w)\in L(w)$ with
 $\theta(w)\nsim \psi(w')$. We claim that for every $v\in V(G)$ the degree of $\theta(v)$ in $ H_\theta$ is at most its capacity. This is true for each
 $v\in F-u$ since by~\eqref{new6} for each neighbor $v'$ of $v$ in $V(G)-F$, $\psi(v')\nsim \phi(v)$. For the same reason, this is true for each
 $x\in V(G)-F-w$. This is true for $u$ by its choice and the fact the only possible neighbor of $\phi(u)$ outside of $\phi(F)$ is $\theta(w)$.
 And this is true for $w$, since $w$ is a $(2;i,j)$-vertex and $\theta(w)\nsim \psi(w')$.

\medskip

\noindent{\bf Case 2:} $w\in X_1$. %In this case, let Denote $S:=X_1\cup Y_1\setminus \{w\}$.
Construct $G'$ from $G-F$ by adding a new vertex $v^*$ adjacent to each vertex in $S-w$. As in Case 1,~\eqref{new5} holds.
Define ${\bf c''}$ by ${\bf c''}(x) = {\bf c}(x)$ for all $x\in V(G')\setminus \{w,v^*\}$,
 ${\bf c''}(v^*) = (0,-1)$ and ${\bf c''}(w) = ({\bf c}_1(w)-1, {\bf c}_2(w)-1)$.
 Define
 $L''$ and $\HH''=(L'',H'')$ exactly as we defined $L'$ and $\HH'=(L',H')$ in Case 1.

Suppose there is $Q\subseteq V(G')$ with $\rho_{G',{\bf c''}}(Q)\leq i-j-1$. If $v^*\notin Q$, then
$w \in Q$, for otherwise $Q\subseteq G$. In this case $\rho_{G,{\bf c}}(Q\cup F)\leq i-j+1+(i-j-1+2)-(i+1)<i-j$. So $v^*\in Q$. Moreover, if $ Q\neq V(G')$, then repeating
the argument of Case 1 we get a contradiction. If $Q= V(G')$, then since $v^*w\notin E(G')$ and $\rho_{G',{\bf c''}}(w)=\rho_{G,{\bf c}}(w)-2$,
 $$
\rho_{G,{\bf c}}(V(G))= \rho_{G,{\bf c}}(F)+\rho_{G,{\bf c}}(Q-v^*)-(i+1)|E_G(F,Q-v^*)|
\leq (i-j+1)+\rho_{G',{\bf c''}}(Q-v^*)$$
$$+2-(i+1)(|E_{G'}(v^*,Q-v^*)|+1)=3+\rho_{G',{\bf c''}}(Q)-(i+1)< i-j,
$$
a contradiction. Thus in all cases $\rho(G',{\bf c''})\geq i-j$. So by the minimality of $G$, $G'$ has a $({\bf c''}, \HH'')$-coloring $\psi$.
As in Case 1, $\psi(v^*)=p(v^*)$ and~\eqref{new6} holds.

 Define an $\HH$-map $\theta$ of $G$ by $\theta(x) = \phi(x)$ for $x\in F$ and $\theta(v) = \psi(v)$ for $v\in V(G)-F$.
 We claim that for every $v\in V(G)$ the degree of $\theta(v)$ in $ H_\theta$ is at most its capacity. If $v\neq w$, then the proof of it is exactly as in Case 1.
 For $v=w$ this follows from the fact that  ${\bf c''}(w) = ({\bf c}_1(w)-1, {\bf c}_2(w)-1)$.
 \end{claim}

\iffalse

\begin{claim}\label{clm:3}
For every $v\in X_1$, ${\bf c}_1(v)= i$, and ${\bf c}_2(v)\geq j+1-i$.
\medskip

\noindent\emph{Proof of Claim~\ref{clm:3}.}
Take any $v\in X_1$. Observe that ${\bf c}_1(v)+{\bf c}_2(v)\geq j+1$. Otherwise we have
\[
\rho_{G,{\bf c}}(F+v)\leq 2(i-j+1)+j-(i+1)= i-j+1.
\]
which is either a contradiction to the maximality of $F$, or to Lemma~\ref{LM-M-2} (ii). Suppose ${\bf c}_1(v)\leq i-1$. Denote the neighbor of $v$ in $F$ by $v_F$. Let $L(v_F)=\{v_1,v_2\}$ such that $v_1\sim p(v)$.
Define ${\bf c'}$ on $V(G)$ by: (a) ${\bf c'}(x) = {\bf c}(x)$ for $x\notin \{v_F,v\}$; (b) ${\bf c'}_1(v) = -1, {\bf c'}_2(v) = {\bf c}_2(v)$; (c) ${\bf c'}(v_i) = {\bf c}(v_i) - 1$ for $i\in \{1,2\}$.
 In $G' = G-vv_F$, if there is $U\subseteq V(G')$ containing $v$ with $\rho_{G',{\bf c'}}(U)<i-j$. Then $v_F\notin U$. Otherwise, $\rho_{G,{\bf c}}(U)\leq \rho_{G',{\bf c'}}(U) + i+1-(i+1)<i-j $. Thus by Lemma~\ref{lem:submod} and the fact that $vv_F$ connects $F$ and $U$, $\rho_{G,{\bf c}}(V(F)\cup U)\leq \rho_{G',{\bf c'}}(U)<i-j$, contradiction. Therefore, $\rho(G',{\bf c'})\geq i-j$ and thus $G'$ is $(i,j,{\bf c'})$-colorable.
 Let $\HH_{G'} = (L,H_{G'})$ be the subcover of $\HH$ induced by $G'$. By the minimality of $G$, there is a $({\bf c'}, \HH_{G'})$-coloring $\phi$ of $G'$. Since ${\bf c'}_1(v) = -1$, $\phi(v) = r(v)$, and by Claim~\ref{clm:2}, $\phi(v_F)=v_1$. Since $v_1\nsim r(v)$, $\phi$ is a $({\bf c},\HH)$-coloring of $G$, contradiction. Hence ${\bf c}_1(v) = i$ and ${\bf c}_2(v)\geq j+1-{\bf c}_1(v) = j+1-i$. \hfill $\bowtie$
\end{claim}
\fi

Let $Q$ be an auxiliary graph with $V(Q)=X$, $E(Q)=Y_0$, where $y\in Y_0$ has end vertices $x_1,x_2$ in $Q$ if $N_G(y)=\{x_1,x_2\}$. From now on, we fix  a $({\bf c}, \HH_F)$-coloring $\psi$ of $F$.

For every $v\in X_1\cup Y_1$, let $v_F$ be its neighbor in $F$ and denote by $\overline{\psi(v)}$ the vertex in $L(v)$ such that $\psi(v_F) \nsim \overline{\psi(v)}$.
%Define $w:X\to\mathbb{Z}$ by:  $w(x)={\bf c}_2(x)-|E(x,Y_1)|$ for $x\in X_0$; $w(x)={\bf c}(\overline{\psi(x)})-|E(x,Y_1)|$ for $x\in X_1$.
%Let $W_\psi$ be the collection of functions $w:X\to\mathbb{Z}$ where $w(x)={\bf c}(\overline{\psi(x)})-|E(x,Y_1)|$ for $x\in X_1$, and $w(x)\in\{ {\bf c}_i(x)-|E(x,Y_1)|: i\in \{1,2\}\}$ for $x\in X_0$. Note that $W_\psi$ is a set determined by the coloring $\psi$ on $F$.
Define function $w_\psi:X\to\mathbb{Z}$ by: $w(x)={\bf c}(\overline{\psi(x)})-|E(x,Y_1)|$ for $x\in X_1$, and $w(x)= {\bf c}_2(x)-|E(x,Y_1)|$ for $x\in X_0$.
Note that $w_\psi$ is determined by the coloring $\psi$ on $F$.

\begin{claim}\label{clm:4.1}
There exists no $({\bf c}, \HH_F)$-coloring $\psi$, such that for every  $A\subseteq X$, % the map $w_\psi$ satisfies
 \begin{equation}\label{eq:idkk}
     \sum_{x\in A}w_\psi(x)\geq |E(Q[A])|.
 \end{equation}
\medskip
\noindent\emph{Proof of Claim~\ref{clm:4.1}.}
Suppose the $({\bf c}, \HH_F)$-coloring $\psi$ satisfies (\ref{eq:idkk}). Define an auxiliary bipartite graph $B=B(V_1, V_2)$ with partite sets  $V_1$ and
$V_2$ where $V_1=Y_0$ and $V_2$ has exactly $w(x)$ copies of $x$ for each $x\in X$. %Note that if $w(x)\leq0$ then $x\notin V_2$.
 For each edge $xy\in E(G)$ with $x\in X, y\in Y_0$, vertex $y$ is adjacent to each copy of $x$ in $B$. For any $S\subseteq V_1$, by (\ref{eq:idkk}),
 $|S|\leq |N_B(S)|$. This means that
$B$  satisfies Hall's condition and hence has  a matching $M$ saturating $V_1$. An orientation $D$ of $Q$ can then be formed as follows: for each pair $x_1, x_2\in X$, and each edge $e$ connecting $x_1, x_2$ in $Q$ ($e$ equals some $y\in Y_0$ such that $y\sim x_1$ and $y\sim x_2$), orient $e$ from $x_1$ to $x_2$ if the edge of $M$ connects $y$ to a copy of $x_1$ in $V_2$, and
 from $x_2$ to $x_1$ otherwise. Then for every $x\in X$, $d^+(x)\leq w(x)$.

Define an $\HH$-map $\phi$ of $G$ as follows.
For every $v\in F$, $\phi(v)=\psi(v)$.
 For every $x\in X_0$,  $\phi(x)=r(x)$. For every $u\in X_1\cup Y_1$,  $\phi(u)=\overline{\psi(u)}$.
For every $y\in Y_0$, if $y = x_1x_2$ in $D$, then choose $\phi(y)\in L(y)$ so that
 $\phi(y)\nsim \phi(x_2)$.

 Let us check that for every $v\in V(G)$,  the degree of $\phi(v)$ in $ H_\phi$ is at most its capacity. This is true for $v\in F$ because $\phi(v)$ has no neighbors in $\phi(V(G)-F)$ by the choice of colors for vertices in $X_1\cup Y_1$. This is true for each $y\in Y$, because each of them has two neighbors
 and $\phi(y)$ has capacity at least $i$. This is true for each $x\in X$ by~\eqref{eq:idkk} and the choice of $w$, $D$ and the colors for the vertices in $Y_0$.
 Thus $\phi$ is a $({\bf c},\HH)$-coloring of $G$, a contradiction.
%\hfill $\bowtie$
\end{claim}

\medskip
By Claim~3, there is  $A\subseteq X$ with
 \begin{equation}\label{new7}
   \sum_{x\in A}w_\psi(x)\leq |E(Q[A])|-1.
    \end{equation}
    Let $Y'\subseteq Y$ consist of the vertices $y\in Y_0$ that have both neighbors in $A$ and the vertices $y\in Y_1$ with a neighbor in $A$. Then
\begin{align*}
    &\,\rho_{G,{\bf c}}(F\cup A\cup Y') = \rho_{G,{\bf c}}(F)+\sum_{x\in A}\big(\rho_{{\bf c}}(x)- |E(x,Y_1)|\big)-|A\cap X_1|(i+1) -|E(Q[A])|\\
    \leq&\,i-j+1+\sum_{x\in A}\big(i-j+1+{\bf c}_1(x)+{\bf c}_2(x)-|E(x,Y_1)|-w_\psi(x)\big)-1-|A\cap X_1|(i+1)\\
    \leq&\,i-j+1+\sum_{x\in A\cap X_0}(i-j+1+{\bf c}_1(x))+\sum_{x\in A\cap X_1}(-j+\max\{{\bf c}_1(x), {\bf c}_2(x)\})-1\leq i-j.
\end{align*}
By Lemma~\ref{LM-M-1}, the equality can hold only if  $A = X, Y' = Y$. Moreover, in this case
every non-strict inequality in the chain above is an equality, and~\eqref{new7} is an equality. The latter yields
 \begin{equation}\label{new8}
\sum_{x\in X}w_\psi(x)= |Y_0|-1,\,\mbox{\em and~\eqref{eq:idkk} holds for every $A\neq X$},
  \end{equation}
  and the former yields
   \begin{equation}\label{new9}
\mbox{\em $\overline{\psi(x)} = p(x)$ for every $x\in X_1$.}
   \end{equation}
%Together with Claim~\ref{clm:3}, we have ${\bf c}(x) = (i,j)$ for all $x\in X_1$.
%In particular,~\eqref{eq:idkk} holds for every $A\neq X$.
We  consider two cases:
\medskip

\noindent{\bf Case 1.} $X_1\neq\varnothing$.
Let $v\in X_1$, and $v_F$ be its neighbor in $F$. Then $\psi(v_F)\sim r(v)$. Define ${\bf c'}$ on $F$  that differs from $ {\bf c}$ only in that ${\bf c'}(\psi(v_F)) = {\bf c}(\psi(v_F))-1$. By (\ref{eq-b}), $\rho(G,{\bf c'})\geq i-j$. By the minimality of $G$,
we can find a $({\bf c'},\HH_F)$-coloring $\psi'$ of $G[F]$. By Claim~\ref{clm:2}, $\psi'(v_F)\neq\psi(v_F)$ and $\overline{\psi'(v)} = r(v)$. Then the above chain of inequalities with $\psi'$ in place of $\psi$ does not satisfy~\eqref{new9},
%$w_{\psi'}$ satisfies (\ref{eq:idkk}),
a contradiction.
%\medskip

\noindent{\bf Case 2.} $X_1=\varnothing$. Then $Y_1\neq \emptyset$. Let $v\in Y_1$, $v_F$ be its neighbor in $F$ and $x'$ be the other neighbor.
 If $\overline{\psi(v)}\nsim r(x')$, then we can define a new function $w'$ that differs from $w$ only in that $w'(x')= {\bf c}_2(x')-|E(x',Y_1)|+1$. Then
 by~\eqref{new8}, $\sum_{x\in X}w'_\psi(x)= |Y_0|$ and~\eqref{eq:idkk} holds with $w'$ in place of $w$ for every $A\neq X$.
 Repeating the proof of Claim~\ref{clm:4.1}, we construct an orientation $D$ of the auxiliary graph $Q$ such that for every $x\in X$, $d^+(x)\leq w'(x)$.
 Then we define a map $\phi$ exactly as in the proof of Claim~\ref{clm:4.1} and check that for every $u\in V(G)$,  the degree of $\phi(u)$ in $ H_\phi$ is at most its capacity almost as in that proof with a change only for $u=x'$: the degree of $r(x')$ does not exceed ${\bf c}_2(x')$ because in addition to other conditions, $\phi(v)\nsim r(x')$. Thus $\overline{\psi(v)}\sim r(x')$.

 Define ${\bf c'}$  as in Case 1. By the same argument, we  can find a $({\bf c'},\HH_F)$-coloring $\psi'$ such that $\psi'(v)\neq \psi(v)$. Now $\overline{\psi'(v)}\nsim r(x')$. This contradicts the previous paragraph.
\end{proof}

\bigskip
Denote the set of $(2;i,j)$-vertices in $G$ by $Y$ and let $X=V(G)\setminus Y$. The proof of the following lemma is very similar to the proof of Claim~\ref{clm:1} in Lemma~\ref{LM-M-3}, so we omit the details.

\begin{lemma}\label{LM-M-4} Both $X$ and $Y$ are independent sets.
\end{lemma}

Given $A\subseteq X$, let $N_2(A)$ denote the set of vertices $v$ in $Y$ such that $|N(v)\cap A|=2$. Let $G_A$ be the subgraph of $G$ induced by $A\cup N_2(A)$. By Lemmas~\ref{LM-M-3} and \ref{LM-M-4}, we have
\begin{equation}\label{eq:bigraph}
\rho_{G,{\bf c}}(G_{X'})=\sum_{v\in X'} \rho_{{\bf c}}(v)-|N_2(X')|>i-j+1 \text{, for every }X'\subsetneq X
\end{equation}
Let $\mathscr{G}$ be the collection of functions $g:Y\to X$ such that $g(y)\in N(y)$ for every $y\in Y$. For each $x\in X$, let $\lambda(x,g)=|g^{-1}(x)|$,
and for each $y\in Y$, let $x_{y,g}$ denote
 the vertex in $N(y)\setminus \{g(y)\}$.

\begin{lemma}\label{lem:gluing}
For every $g\in\mathscr{G}$ there is $x\in X$ such that $\lambda(x,g)> {\bf c}_2(x)$.
\end{lemma}
\begin{proof}
Suppose some $g\in\mathscr{G}$ satisfies $\lambda(x,g)\leq {\bf c}_2(x)$ for every $x\in X$. Define an $\HH$-map $\phi$ by: $\phi(x)=r(x)$ for every $x\in X$, $\phi(y)\nsim \phi(x_{y,g})$ for every $y\in Y$. Then $\phi$ is a $({\bf c}, \HH)$-coloring of $G$, a contradiction.
\end{proof}

Let $\widehat{G}$ be an auxiliary multigraph, where $V(\widehat{G})=X$ and $E(\widehat{G})=Y$, such that for every $a,b\in V(\widehat{G})$, each $y\in N(a)\cap N(b)$ corresponds bijectively to an edge between $a$ and $b$.
 Let $\mathscr{D}$ be the collection of all digraphs obtained by orienting the edges of $\widehat{G}$. Define the bijection
\begin{align*}
    \pi: \mathscr{G}\to\mathscr{D},
\end{align*}
so that for every $g\in \mathscr{G}$, the head of each edge $y\in E(\widehat{G})$ in $\pi(g)$ is $g(y)$.
%For any $A\subseteq V(\widehat{G})$, let $\partial^j A:= \{v\in V(\widehat{G})\setminus\partial^{j-1} A : vu\in E(\pi(g))\text{ for some }u\in \partial^{j-1} A \}$, where $\partial^0 A := A$. Let $A^\infty:= \cup_{j = 0}^\infty \partial^j A$.
For $g\in \mathscr{G}$, let $S_g := \{v\in X : \lambda(v,g)>{\bf c}_2(v) \}$ and let $S'_g$ be
 the set of the vertices $v\in X$ such that $\pi(g)$ has a directed path from  $v$ to $S_g$. By definition, $S'_g\supseteq S_g$, and by Lemma \ref{lem:gluing}, $|S_g|\geq1$ for all $g$.

Now we fix  $g_0\in \mathscr{G}$ such that
\begin{equation}\label{eq:ming}
\sum_{v\in X}\max\{0,\lambda(v,g_0)-{\bf c}_2(v)\}=\min_{g\in\mathscr{G}}\sum_{v\in X}\max\{0,\lambda(v,g)-{\bf c}_2(v)\}.
\end{equation}
%Let $S_{g_0} := \{v\in X : \lambda(v,g_0)>{\bf c}_2(v) \}$. By Lemma \ref{lem:gluing}, $|S_{g_0}|\geq1$.

%Let $S_{g_0}^\infty=\bigcup_{t=0}^\infty\partial^tS_{g_0}$, where $v\in \partial^j S_{g_0}$ if there is $u\in \partial^{j-1}S_{g_0}$ such that $vu\in E(\pi(g_0))$, and $\partial^0 S_{g_0}=S_{g_0}$.

\begin{lemma}\label{lem:s}
$|S_{g_0}'|\geq2$.
\end{lemma}
\begin{proof}
Suppose $S_{g_0}'=\{v\}$. Then by Lemma \ref{lem:gluing}, $S_{g_0}=\{v\}$ and $d^-_{\pi(g_0)}(v)=0$.
This means $\lambda(v,g_0)=0$, thus by the definition of  $S_{g_0}$, ${\bf c}_2(v)=-1$, a contradiction to Lemma \ref{LM-M-2}.
\end{proof}

\begin{lemma}\label{lem:A}
For every $x\in S_{g_0}'$, $\lambda(x,g_0)\geq {\bf c}_2(x)$.
\end{lemma}
\begin{proof}
Suppose there is $v_0\in S_{g_0}'$ such that $\lambda(v_0,g_0)<{\bf c}_2(v_0)$. Let $P:=v_0v_1\dots v_t$ be a $v_0,v_t$-path, where  $v_t\in S_{g_0}$. Obtain $D\in\mathscr{D}$ from $\pi(g_0)$ by reversing all edges in $P$, and denote $\pi^{-1}(D)$ by $h$. Then
\[
\sum_{v\in X}\max\{0,\lambda(v,h)-{\bf c}_2(v)\}=\sum_{v\in X}\max\{0,\lambda(v,g_0)-{\bf c}_2(v)\}-1,
\]
contradicting (\ref{eq:ming}).
\end{proof}
\medskip

We say that a vertex $y\in Y$ adjacent to $u$ and $v$ in $G$ is \emph{even} (with respect to $\HH$), if in $H$, each vertex in $L(y)$ is adjacent either to both rich vertices in $L(u)\cup L(v)$, or to both poor vertices in $L(u)\cup L(v)$. Otherwise, we say $y$ is \emph{odd}.

%We first consider the case where $|S_{g_0}|=1$.

\begin{lemma}\label{lem:SS}
For each $v\in S'_{g_0}$,  $\lambda(v,g_0)\leq {\bf c}_1(v)+{\bf c}_2(v)+1$.

%Suppose $S_{g_0}=\{v\}$. Then $\lambda(v,g_0)\geq {\bf c}_1(v)+{\bf c}_2(v)+2$.
\end{lemma}
\begin{proof} Suppose there is $v\in S'_{g_0}$ with  $\lambda(v,g_0)\geq {\bf c}_1(v)+{\bf c}_2(v)+2$.
By Lemma \ref{lem:A} and the definition of $S'_{g_0}$,
\[
2+{\bf c}_1(v)+\sum_{x\in S_{g_0}'}{\bf c}_2(x)\leq\sum_{x\in S_{g_0}'}\lambda(x,g_0)=|N_2(S_{g_0}')|.
\]
Therefore, by Lemma~\ref{lem:s},
\begin{align*}
\rho_{G,{\bf c}}(G_{S_{g_0}'})&=(i-j+1)|S_{g_0}' |+\sum_{x\in S_{g_0}' }{\bf c}_1(x)+\sum_{x\in S_{g_0}' }{\bf c}_2(x)-|N_2(S_{g_0}' )|\\
&\leq (i-j+1)|S_{g_0}' |+(|S_{g_0}' |-1)i-2\\
&=i-j-1+(|S_{g_0}' |-1)(2i-j+1)\leq i-j-1,
\end{align*}
which contradicts (\ref{eq:bigraph}) or the choice of $G$.
\end{proof}

\begin{lemma}\label{lem:S}
$|S_{g_0}|>1$.
%Suppose $S_{g_0}=\{v\}$. Then $\lambda(v,g_0)\geq {\bf c}_1(v)+{\bf c}_2(v)+2$.
\end{lemma}
\begin{proof}
Suppose $S_{g_0}=\{v\}$.
Let $\phi_1$, $\phi_2$ be $\HH$-maps defined by: $\phi_1(x)=\phi_2(x) =\ r(x)$ for $x\in X \setminus \{v\}$, $\phi_1(v) = r(v), \phi_2(v) = p(v)$; for every $y\in Y$,  $\phi_1(y)\nsim \phi_1(x_{y,g_0})$ and $\phi_2(y)\nsim \phi_2(x_{y,g_0})$. If $g_0^{-1}(v)$ contains at most ${\bf c}_2(v)$ odd vertices, then
$\phi_1$ is a $({\bf c}, \HH)$-coloring, a contradiction. Similarly, if  $g_0^{-1}(v)$ contains at most ${\bf c}_1(v)$ even vertices, then
$\phi_2$ is a $({\bf c}, \HH)$-coloring.
%Since neither $\phi_1$ nor $\phi_2$ is a $({\bf c}, \HH)$-coloring,
Hence $g_0^{-1}(v)$ contains at least ${\bf c}_1(v)+1$ even vertices and at least ${\bf c}_2(v)+1$ odd vertices. Thus $\lambda(v,g_0)\geq {\bf c}_1(v)+{\bf c}_2(v)+2$. This contradicts Lemma~\ref{lem:SS}.
\end{proof}

\medskip
%We now consider the case where $|S_{g_0}|>1$. The proof goes by induction on $|S_{g_0}|$.
%For every $g\in \mathscr{G}$, let $Y_{S_{g}}=Y\cap N_2(S_{g})$.
For $v\in X$,
let $\lambda_S(v,g):=|g^{-1}(v)\cap N_2(S_{g})|$, and $\lambda_{\overline{S}}(v,g):=|g^{-1}(v)\cap (Y\setminus N_2(S_{g}))|$. Thus $\lambda(v,g)=\lambda_S(v,g)+\lambda_{\overline{S}}(v,g)$ for every $v\in X$. The proof goes by induction on $|S_{g_0}|$. Lemma~\ref{lem:S} provides
the base of induction. Now we do an induction step.
%Suppose the theorem holds for all pairs $(G,g_0)$ with $|S_{g_0}|$

If for some  $v\in S_{g_0}$, $\lambda_S(v,g_0)\geq \lambda(v,g_0)-{\bf c}_2(v)$,
 consider the digraph $D$ obtained from $\pi(g_0)$ by reversing $\lambda(v,g_0)-{\bf c}_2(v)$ edges in $\pi(g_0)[S_{g_0}]$ each with head vertex $v$, and let $h=\pi^{-1}(D)$. Then
 $\sum_{v\in X}\max\{0,\lambda(v,g_0)-{\bf c}_2(v)\}=\sum_{v\in X}\max\{0,\lambda(v,h)-{\bf c}_2(v)\}$
 and $S_h=S_{g_0}\setminus\{v\}$, so $|S_h|=|S_{g_0}|-1$. Hence by induction assumption, the theorem holds. Thus we may assume that
\begin{equation}\label{new10}
\mbox{\em for every $v\in S_{g_0}$, $\lambda_S(v,g_0)<\lambda(v,g_0)-{\bf c}_2(v)$.}
\end{equation}

Let $Y_e$ be the set of even vertices in $Y$, and  $Y_o$ be the set of odd vertices in $Y$. For every $g\in\mathscr{G}$ and every $v\in X$, let $\lambda^e_{\overline{S}}(v,g)=|g^{-1}(v)\cap (Y^e\setminus N_2(S_{g}))|$ and  $\lambda^o_{\overline{S}}(v,g)=|g^{-1}(v)\cap (Y^o\setminus N_2(S_{g}))|$,
 so for every $v\in X$, $\lambda_{\overline{S}}(v,g)=\lambda^e_{\overline{S}}(v,g)+\lambda^o_{\overline{S}}(v,g)$.
Define
$$T_{g_0}=\{v\in S_{g_0}\, :\,\lambda_{\overline{S}}^e(v,g_0)+\lambda_S(v,g_0)\geq {\bf c}_1(v)+1\}\qquad \mbox{and}$$
$$R_{g_0}=\{u\in S_{g_0}\, :\, \lambda_{\overline{S}}^o(u,g_0)+\lambda_S(u,g_0)\geq {\bf c}_2(u)+1\}.$$

%Let $T_{g_0},R_{g_0}\subseteq S_{g_0}$, such that for every $v\in T_{g_0}, u\in R_{g_0}$, $\lambda_{\overline{S}}^e(v,g_0)+\lambda_S(v,g_0)\geq {\bf c}_1(v)+1, \lambda_{\overline{S}}^o(u,g_0)+\lambda_S(u,g_0)\geq {\bf c}_2(u)+1$.

\iffalse
\begin{lemma}\label{ma-new-lm3}
Neither $R_{g_0}$ nor $T_{g_0}$ is empty.
\end{lemma}
\begin{proof}
Suppose $R_{g_0} = \varnothing$, then for all $u\in S_{g_0}$, $\lambda_{\overline{S}}^o(u,g_0)+\lambda_S(u,g_0)\leq {\bf c}_2(u)$. Let $\phi$ be a representative map with $\phi(v) = r(v)$ for every vertex $v\in X$ and for all $y\in Y$, $\phi(y)$ is not adjacent to the image of its tail.

Suppose $T_{g_0} = \varnothing$, then for all $v\in S_{g_0}$, $\lambda_{\overline{S}}^e(v,g_0)+\lambda_S(v,g_0)\leq {\bf c}_1(v)$. Let $\psi$ be a representative map with $\psi(v) = p(v)$ for every vertex $v\in X$ and for all $y\in Y$, $\psi(y)$ is not adjacent to the image of its tail.

(will write later that $\phi$ and $\psi$ are colorings)
\end{proof}
\fi

\begin{lemma}\label{ma-new-lm4}
$R_{g_0}\cap T_{g_0} \neq \varnothing$.
\end{lemma}
\begin{proof}
Suppose $R_{g_0}\cap T_{g_0} = \varnothing$.
Let $\phi$ be an $\HH$-map such that $\phi(v) = r(v)$ for $v\in X\setminus R_{g_0}$, $\phi(x) = p(x)$ for $x\in R_{g_0}$, and $\phi(y)\nsim \phi(x_{y,g_0})$ for $y\in Y$.

Let us check that the degree of $\phi(v)$ in $ H_\phi$ is at most ${\bf c}(\phi(v))$ for each $v\in V(G)$. This is true for each $y\in Y$ because $\phi(y)$ has at most one neighbor in $ H_\phi$, and ${\bf c}_2(y)>{\bf c}_1(y)=i\geq 3$. This is true for each $x\in X-S_{g_0}$ because
$\phi(x)$ has at most $\lambda(x,g_0)$ neighbors in $ H_\phi$, ${\bf c}(\phi(x))={\bf c}_2(x)$, and $\lambda(x,g_0)\leq {\bf c}_2(x)$.
Suppose $v\in S_{g_0}\setminus R_{g_0}$. Then  $\phi(v) = r(v)$, so
${\bf c}(\phi(v))={\bf c}_2(v)$. If $\phi(y)$ is a neighbor of $\phi(v)$ in $\HH$, and the  neighbor $x$ of $y$ in $G$ distinct from $v$ is not  in $S_{g_0}$,
then $\phi(x)=r(x)$, and $\phi(y)\nsim \phi(x)$. Hence in order $\phi(y)$ to be a neighbor of $\phi(v)$, vertex $y$ needs to be odd. The total number of such neighbors is $\lambda_{\overline{S}}^o(v,g_0)$. By the definition of $R_{g_0}$,  $\lambda_{\overline{S}}^o(v,g_0)+\lambda_S(v,g_0)\leq {\bf c}_2(v)$,
thus our claim holds for $v$. Finally, if $u\in  R_{g_0}$, then  $\phi(u) = p(u)$ and
${\bf c}(\phi(u))={\bf c}_1(u)$. Symmetrically to above, the total number of neighbors $\phi(y)$  of $\phi(u)$ in $\HH$ such that the
neighbor $x$ of $y$ in $G$ distinct from $u$ is not  in $S_{g_0}$ is $\lambda_{\overline{S}}^e(u,g_0)$. Since $u\notin T_{g_0}$,
 $\lambda_{\overline{S}}^e(u,g_0)+\lambda_S(u,g_0)\leq {\bf c}_1(u)$.
Thus, $\phi$ is a $({\bf c},\HH)$-coloring of $G$, a contradiction.
\end{proof}

By definition, for every $v\in T_{g_0}\cap R_{g_0}$,
$$( {\bf c}_1(v)+1)+( {\bf c}_2(v)+1)\leq (\lambda_{\overline{S}}^e(v,g_0)+\lambda_S(v,g_0))+
(\lambda_{\overline{S}}^o(v,g_0)+\lambda_S(v,g_0))
$$
$$=\lambda_{\overline{S}}(v,g_0)+2\lambda_S(v,g_0)=\lambda(v,g_0)+\lambda_S(v,g_0).
$$
By~\eqref{new10}, this is at most $2\lambda(v,g_0)- {\bf c}_2(v)-1$.
Therefore, for every $v\in T_{g_0}\cap R_{g_0}$,
\begin{equation}\label{eq:key}
    \lambda(v,g_0)\geq {\bf c}_2(v)+\frac{{\bf c}_1(v)}{2}+\frac{3}{2}.
\end{equation}
%This follows from previous Lemmas and the fact that $\lambda_S(x,g_0)<\lambda(x,g_0)-{\bf c}_2(x)$.

\begin{lemma}\label{notreallynew-lm5}
$|R_{g_0}\cap T_{g_0}|\leq 1$.
\end{lemma}
\begin{proof}
Suppose there are distinct $u,v\in R_{g_0}\cap T_{g_0}$. By (\ref{eq:key}) and Lemma~\ref{lem:A},

\begin{align*}
 |N_2(S_{g_0}' )|&=\sum_{x\in S_{g_0}' }\lambda(x,g_0)
 =\lambda(u,g_0)+\lambda(v,g_0)+\sum_{x\in S_{g_0}' \setminus S_{g_0}}\lambda(x,g_0)+\sum_{x\in S_{g_0}\setminus \{u,v\}}\lambda(x,g_0)\\
 &\geq |S_{g_0}|-2+\sum_{x\in S_{g_0}' }{\bf c}_2(x)+\frac{{\bf c}_1(u)+{\bf c}_1(v)}{2}+3.
\end{align*}
Therefore,
\begin{align*}
\rho_{G,{\bf c}}(S_{g_0}' )&=(i-j+1)|S_{g_0}' |+\sum_{x\in S_{g_0}' }{\bf c}_1(x)+\sum_{x\in S_{g_0}' }{\bf c}_2(x)-|N_2(S_{g_0}' )|\\
&\leq (i-j+1)|S_{g_0}' |+(|S_{g_0}' |-1)i-|S_{g_0}|-1\\
&\leq i-j-1+(|S_{g_0}' |-1)\big(2i-j+1\big)-1\leq i-j-2,
\end{align*}
contradicting the choice of $G$.
\end{proof}

Now the only remaining case is that  $|R_{g_0}\cap T_{g_0}| = 1$. Let $R_{g_0}\cap T_{g_0} = \{v\}$.

Define  $\HH$-maps $\phi_1$ and $\phi_2$  as follows:
$\phi_1(x)=\phi_2(x) =\ r(x)$ for every $x\in X\setminus R_{g_0}$, $\phi_1(x)=\phi_2(x) = p(x)$ for all $x\in R_{g_0}\setminus\{v\}$,
$\phi_1(v) = r(v), \phi_2(v) = p(v)$, and for every $y\in Y$,  $\phi_1(y)\nsim \phi_1(x_{y,g_0})$ and $\phi_2(y)\nsim \phi_2(x_{y,g_0})$.

 If $g_0^{-1}(v)$ contains at most ${\bf c}_1(v)$ even vertices, then repeating the proof of Lemma~\ref{ma-new-lm4} we conclude that
$\phi_1$ is a $({\bf c}, \HH)$-coloring, a contradiction. Similarly, if  $g_0^{-1}(v)$ contains at most ${\bf c}_2(v)$ odd vertices, then
$\phi_2$ is a $({\bf c}, \HH)$-coloring. So we have $\lambda(v,g_0)\geq {\bf c}_1(v)+{\bf c}_2(v)+2$, this contradicts Lemma~\ref{lem:SS} completing the proof of the theorem.

\section{Constructions }
In this section, we construct $(i,j)$-critical graphs with $i\geq 3, j\geq 2i+1$ that attain equality of the upper bound in Theorem~\ref{MT-F}.
We first define flags, which will be used to control the capacity of the vertices.
\begin{definition}[flags]
Given a vertex $v$, a {\em flag at $v$} is a graph containing $i+1$ many degree $2$ vertices $\{u_1,\dots,u_{i+1}\}$ and a degree $(i+2)$ vertex $x$, such that all of these vertices are adjacent to $v$, and $x$ is adjacent to all the vertices in $\{u_1,\dots,u_{i+1}\}$. See Figure~\ref{fig:flag}. $x$ is called the \emph{top vertex} in this flag, $v$ is the \emph{base vertex} of the flag, and $u_1,\dots,u_{i+1}$ are \emph{middle }vertices.

In the cover graph (we abbreviate 'the flag-induced cover graph' here by 'flag'), we say that a flag (with base vertex $v$, top vertex $x$, middle vertices $u_1,\dots, u_{i+1}$) is {\em parallel} if $p(x)\sim p(v), r(x)\sim r(v)$ and $u_t$ is even for every $t$; when $p(x)\sim r(v), r(x)\sim p(v)$ and $u_t$ is odd for every $t$, we call such flag a {\em twisted} flag.
\end{definition}
\begin{figure}[h]
    \centering
    \includegraphics[width=1.7in]{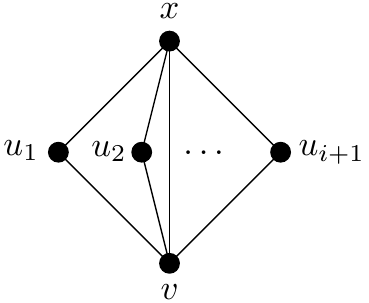}
\caption{A flag at vertex $v$.}\label{fig:flag}
\end{figure}
The following observation about flags is easy to check by hand.
\begin{claim}\label{clm0}
{\em Let $\HH=(H,L)$ a $2$-fold cover of a graph $G$ and
$F$ be a flag with base $v$ in $\HH$. Let
$\phi$ be a coloring of $H-(V(F)-v)$.
If $F$ is parallel and $\phi(v)=r(v)$, then in any extension of $\phi$ to $F$, $\phi(v)$ will have a neighbor in $\phi(F)$, and there is an extension in which $\phi(v)$ will have exactly one neighbor in $\phi(F)$. Similarly, if $F$ is twisted and $\phi(v)=p(v)$, then in any extension of $\phi$ to $F$, $\phi(v)$ will have a neighbor in $\phi(F)$, and there is an extension in which $\phi(v)$ will have exactly one neighbor in $\phi(F)$. In all other cases, we can extend $\phi$ to
$F$ so that  $\phi(v)$ will have no neighbors in $\phi(F)$.
}
\end{claim}

 Hence, adding a parallel flag on a vertex $v$ essentially decreases ${\bf c}_2(v)$ by $1$, and adding a twisted flag on $v$ essentially decreases ${\bf c}_1(v)$ by $1$.

Given $m\geq1$, we now construct the graph $G_m$. When $m\geq2$, let $G_m$ be obtained from a path $v_1\dots v_m$, by adding $i+1$ flags to $v_1$, adding $i$ flags to $v_t$ for every $1<t<m$, and adding $i+j+1$ flags to $v_m$. When $m=1$, we define $G_1$ as a single base vertex with $i+j+2$ flags. See Figure~\ref{fig:ij}.
\begin{figure}[h]
    \centering
    \includegraphics[width=5.9in]{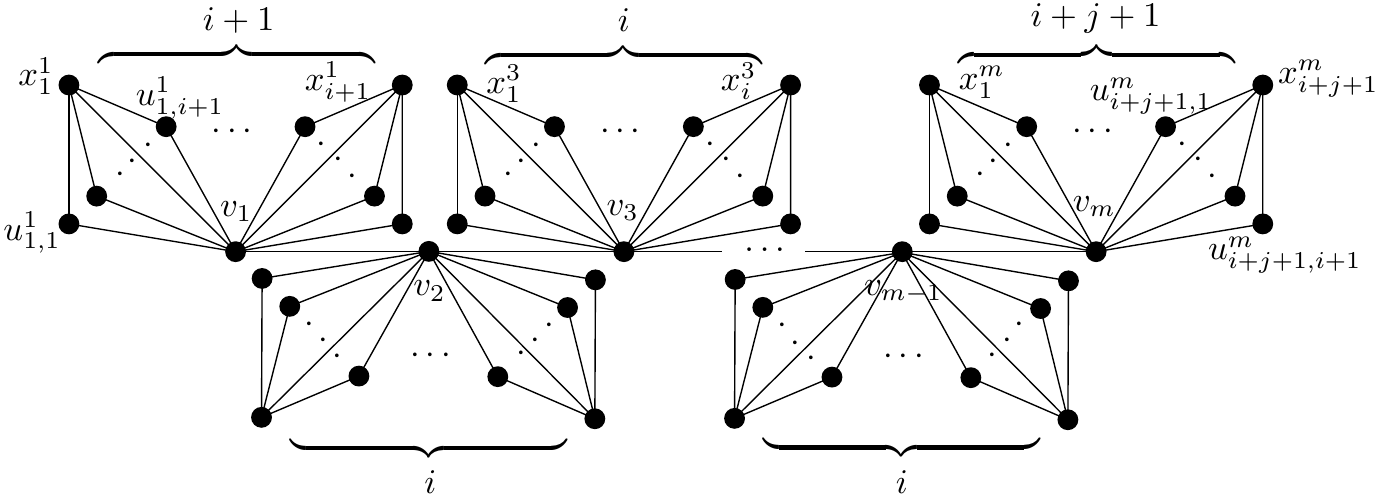}
\caption{Critical graphs $G_m$ for $(i,j)$-colorings.}\label{fig:ij}
\end{figure}

Note that for every $m\geq1$, $|V(G_m)|=(i+2)(mi+j+2)+m$ and $|E(G_m)|=(2i+3)(mi+j+2)+m-1$, thus
\[
|E(G_m)|=\frac{(2i+1)|V(G_m)|+j-i+1}{i+1}.
\]
\begin{proposition}
Let $i\geq1$, $j\geq 2i+1$ be integers. Then $G_m$ is $(i,j)$-critical for every $m$.
\end{proposition}
\begin{proof}
We first construct for each $m$ a $2$-fold cover $\HH_m = (L_m, H_m)$ of $G_m$, such that no $\HH_m$-map is an $\HH_m$-coloring.
When $m = 1$, let $i+1$ flags be twisted and the remaining $j+1$ flags be parallel. When $m\geq 2$, let $j$ flags based on $v_m$ be parallel, and the remaining flags in $G_m$ be twisted. For the path $v_1\cdots v_m$ in $H_m$, let $r(v_t)\sim p(v_{t+1})$ for $t = 1,\dots, m-2$, and $r(v_{m-1})\sim r(v_m)$.
Suppose $\phi$ is an $\HH_m$-coloring of $G_m$. By Claim~\ref{clm0}, $\phi(v_1) = r(v_1)$. Since $p(v_2)\sim r(v_1)$, and there are $i$ twisted flags based $v_2$, $\phi(v_2)$ has to be $r(v_2)$. Similarly, $\phi(v_t) = r(v_t)$ for all $t = 1,\dots,m-1$. Now since there are $i+1$ twisted flags based on $v_m$, by Claim~\ref{clm0}, $\phi(v_m)$ cannot be $p(v_m)$. But then again by Claim~\ref{clm0}, $\phi(v_m)$ has   $j$ neighbors from the parallel flags plus $\phi(v_{m-1})=r(v_{m-1})$ also is its  neighbor, a contradiction.

\iffalse
We first construct a cover graph $\HH_m$ of $G_m$ such that $\HH_m$ is not $(i,j)$-colorable.
When $m=1$, define $\HH_1$ such that $i+1$ flags are twisted and the remaining $j+1$ flags are parallel.
When $m\geq 2$, define $\HH_m$ where the flags based on $v_1,\dots, v_{m-1}$ and $i+1$ flags based on $v_m$ are twisted, the remaining $j$ flags based on $v_m$ are parallel. We further require $p(v_t)$ is adjacent to $r(v_{t+1})$ for $t = 1,\dots, m-2$, and $p(v_{m-1})$ is adjacent to $r(v_m)$. Suppose $\phi$ is a coloring of $G_m$ in $\HH_m$. Since $v_1$ has $i+1$ twisted flags, each of them will decrease the capacity of $c_1(v_1)$ by $1$, thus $\phi(v_1)=r(v_1)$. Similarly, since $p(v_2)$ is adjacent to $r(v_1)$, and $v_2$ has $i$ twisted flags, $\phi(v_2)=r(v_2)$. By using a inductive argument, we have that  $\phi(v_{m-1})=r(v_{m-1})$. Since $v_m$ has $i+1$ twisted flags, $\phi(v_m)=r(v_m)$. But $v_m$ also has $j$ parallel flags, and $r(v_m)$ is adjacent to $r(v_{m-1})$, thus the degree of $r(v_m)$ in the representative map $\phi$ is $j+1$, this contradicts $\phi$ is an $(i,j)$-coloring.
\fi

\medskip
We now show that every proper subgraph of $G_m$ is $(i,j)$-colorable. It suffices to show that $G_m-e$ is $(i,j)$-colorable for any $e\in E(G_m)$.
\begin{claim}\label{clm}
{\em Let $F$ be obtained by removing an edge $e$ from a flag with base $v$. Let $\HH = (L,H)$ be a $2$-fold cover of $F$. Then for each of the choices
$\phi(v) = p(v)$ and $\phi(v) = r(v)$,
there is an $\HH$-map $\phi$, such that  the degree of $\phi(v)$ in $H_\phi$ is  $0$.}

%If $F'$ is a proper subgraph of a flag with base $v$, let $\HH_{F'}$ be an arbitrary cover graph of $F'$. For every $v^*\in L(v)$, let $\phi(v)=v^*$. Then $\phi$ can be extended to a coloring of $F'$ on $\HH_{F'}$, such that $\phi(v)$ has degree $0$ in the $\phi$-representative map.
\medskip

\noindent\emph{Proof of Claim~\ref{clm}.}
Denote the top vertex by $x$. If $x\nsim v$, define $\phi(x) = r(x)$ and for each middle vertex $u$, define $\phi(u)$ so that $\phi(u)\nsim \phi(v)$. Then $\phi$ is a desired $\HH$-map.
Now assume $x\sim v$. Choose $\phi(x)\nsim \phi(v)$. If $e = xu_t$ for some middle vertex $u_t$, then let $\phi(u_t)\nsim \phi(v)$; if $e = u_tv$, let $\phi(u_t)\nsim \phi(x)$
 In either case, $\phi(u_t)$ is adjacent to neither $\phi(x)$ or $\phi(v)$. For the remaining middle vertices, choose $\phi(u_k)\nsim \phi(v), k\neq t$. There are only $i$ such $u_k$'s, thus $\phi$ is a desired $\HH$-map.
%\hfill $\bowtie$
%If $x$ is not adjacent to $v$ in $F'$, for every middle vertex $u$, let $\phi(u)$ be the vertex in $L(u)$ which is not adjacent to $\phi(v)$, and let $\phi(x)=r(x)$, then $\phi$ is a coloring of $F'$ which satisfies the condition. Now, suppose that $x$ is adjacent to $v$. For every middle vertex $u$, if $uv$ is an edge in $F'$, let $\phi(u)$ be the vertex in $L(u)$ which is not adjacent to $\phi(v)$, and let $\phi(x)$ be the vertex in $L(x)$ which is not adjacent to $\phi(v)$. If $u$ is not adjacent to $v$, let $\phi(u)$ be not adjacent to $\phi(x)$. Then we get a coloring $\phi$ of $F'$ satisfies the condition.
\end{claim}

Claim~\ref{clm} essentially says that removing an edge from a flag is `equivalent' (with respect to coloring) to removing the whole flag. Hence $G_1-e$ contains either at most $i$ twisted flags, or at most $j$ parallel flags. In either case $G_1-e$ is colorable.

Let $m\geq 2$ and $\HH = (L,H)$ be a $2$-fold cover of $G_m-e$. We will construct an $\HH$-map $\phi$.
If for a cover $\HH' = (L',H')$ of $G_m$,  there are at most $i$ twisted flags on $v_m$ in $H'$, then we can define an $\HH'$-map $\phi'$ of $G_m$ by $\phi'(v_m) = p(v_m)$ and $\phi'(v_k) = r(v_k)$ for all $k\neq m$.
Since all the possible neighbors of  $\phi'(v_m)$ in $H'_{\phi'}$ will be from the twisted flags based on $v_m$, the degree of $\phi'(v_m)$ will not exceed $i$. For $k = 1,\dots, m-1$, since $j> i+1$, the degree of $\phi'(v_k)$ will not exceed $j$.
%Then $\phi$ can be extended to a coloring of $G_m$.
Hence
\begin{equation}\label{blue}
\mbox{\em
we  consider only covers   of $G_m-e$ with at least $i+1$ twisted flags on $v_m$.}
\end{equation}

{\bf Case 1:} $e$ is belongs  to some flag $F$. If $F$ is based on $v_m$, then by~\eqref{blue}  there are at most $j-1$ parallel flags on $v_m$. Let $\phi(v_k) = r(v_k)$ for each $k$.
Then by Claims~\ref{clm0} and~\ref{clm}, we can extend $\phi$ to each of the flags so that the degree of $\phi(v_k)$ in $H'_\phi$ will be at most $j$ for each $k$.
%Then $\phi$ can be extended to a coloring.

If $F$ is based on $v_t$ for some $t\neq m$, then there are at most $i-1$ twisted or parallel flags based on $v_t$ when $t>1$ and at most $i$ such
flags based on $v_t$ when $t=1$. Let $\phi(v_k) = r(v_k)$ for $k \in \{ 1,\dots, t-1,m\}$, and $\phi(v_l) = p(v_l)$ for $l = t,\dots, m-1$.
%When $F$ is based on $v_1$, let $\phi(v_k) = p(v_k)$ for each $k\neq m$.
Again by Claims~\ref{clm0} and~\ref{clm}, we can extend $\phi$ to each of the flags so that
for all $k$, the degree of $\phi(v_k)$ in $H'_\phi$ is at most $i+1<j$, for each $k = 1,\dots, t-1$, and the degree of $\phi(v_k)$ in $H'_\phi$ is at most $i$ for each $k = t,\dots, m-1$.  Moreover, by~\eqref{blue} we can provide that the degree of $\phi(v_m)$ in $H'_\phi$ is at most $j$.
Thus in all cases, $\phi$ can be extended to an $\HH$-coloring.

{\bf Case 2:}
 $e = v_tv_{t+1}$ for some $t \in \{1,\dots, m-1\}$.  Let $\phi(v_k) = r(v_k)$, for each $k \in \{1,\dots, t,m\}$,  and $\phi(v_k) = p(v_k)$ for each $k = t+1,\dots, m-1$. Similarly to Case 1, $\phi$ again can be extended to an $\HH$-coloring of $G_m-e$.
\end{proof}

\end{document}